\documentclass{article}
\usepackage{amsthm, amsfonts, url}
\usepackage{amscd,amsmath,amssymb}

\newtheorem{theorem}{Theorem}

\newtheorem{corollary}{Corollary}

\title{Globalization of the partial isometries of metric spaces
 and local approximation of the group of isometries of Urysohn space.}
\author{A.~M.~Vershik\thanks{St.~Petersburg Department of
Steklov Institute of Mathematics, 27 Fontanka, 191023
St.~Petersburg, Russia. E-mail: {\tt vershik@pdmi.ras.ru} The paper
partially supported by grants of RFBR 05-01-00899 and NSh
4329.2006.1}}

\begin{document}

\maketitle \tableofcontents
\bigskip
\bigskip
\bigskip
\begin{centerline}
{\textbf{Abstract}}
\end{centerline}
{We prove the equivalence of the two important facts about finite metric spaces and universal
Urysohn metric spaces $\Bbb U$, namely theorem A and theorem B below: Theorem A (Approximation):
The group of isometry $ISO(\Bbb U)$ contains everywhere dense locally finite subgroup; Theorem
G(Globalization): For each finite metric space $F$ there exists another finite metric space $\bar
F$ and isometric imbedding $j$ of $F$ to $\bar F$ such that isometry $j$ induces the imbedding of
the group monomorphism of the group of isometries of the space $F$ to the group of isometries of
space $\bar F$ and each partial isometry of $F$ can be extended up to global isometry in $\bar F$.
The fact that theorem $G$, is true was announced in 2005 by author without proof, and was proved by
S.Solecki in \cite{Sol} (see also \cite{P,P1}) based on the previous complicate results of other
authors. The theorem is generalization of the Hrushevski's theorem about the globalization of the
partial isomorphisms of finite graphs. We intend to give a constructive proof in the same spirit
for metric spaces elsewhere. We also give the strengthening of homogeneity of Urysohn space and in
the last paragraph we gave a short survey of the various constructions of Urysohn space including
the new proof of the construction of shift invariant universal distance matrix from \cite{CV}}.

\section{Introduction. About Urysohn space.}

  The first time when I have heard about Urysohn space was in 90-th when my coauthor V.Berstovsky
  (during the preparation of our paper \cite{BV}) told me that he had read in his student time
  the popular article in a Journal (later I found out that it was article by P.S.Alexandrov in
  the Russian Popular Journal "Quant"  $\#8$,1974 devoted to 50-th anniversary of P.S.Urysohn death)
  the short mentioning of the Urysohn space. We quoted in our paper the Urysohn article \cite{U}
  and the link between Urysohn space and Gromov distance between metric spaces.
  I was amazed by the beauty of the Urysohn's construction in \cite{U} (it was republished
  in the Russian volume of his collected works) and was very surprised that
  during almost 70 years only about 10 papers in mathematical literature concerned of that remarkable space.
  Then I tried to connect Urysohn construction 1)with notion of randomness in the theory of metric spaces
  and other random objects which about I had interested with, and 2)with dynamical systems. I have written
  about this in the short paper \cite{V^0}. I tried to popularize the Urysohn space among my colleagues
  and students, and suggested to M.Gromov to include the information about it to the 2-nd edition of his
  remarkable book \cite{G}; he indeed put very interesting remarks on this. I had remarked that
  in the space of all metrics on the naturals numbers $\Bbb N$ equipped with usual (weak) topology
  the set of those metrics under which completion of $\Bbb N$ is Urysohn space (up to isometry) - is everywhere
  dense $G_{\delta}$ set. Later in 2000 in Barcelona on 3-d European Congress of Mathematics I have heard
  the talk by P.Cameron on the universal graphs (I knew nothing about universal graph before that)
  and told him that random graphs looks like simple variant (baby model) of the random metric space which is in a sense
  Urysohn space. Some papers by P.Cameron was very tightly connected with Urysohn space but it was
  understood completely only later when we had constructed in \cite{CP} the Toeplitz universal distance
  matrix; this means that Urysohn space can have an isometry with dense orbit and consequently
  this space has a (nonunique) structure of monothetic (abelian) group with invariant metric
  (see the last paragraph of this paper).

    My impression that the bunch of the problems connected with Urysohn space is extremely wide and
  the various properties of it give the material for the interesting generalization. One of the fresh example -
  the notion of linear rigid space which is introduced in basing on the remarkable observation of
  R.Holms \cite{Ho}: Urysohn space has only one up to isometry dense embedding to the Banach space.
  It is interesting that many problems about finite metric space had been arisen after analysis of
  Urysohn space. Another class of the problems concern of the detail structure of the Urysohn metric and
  transparent model of the space. Finally the connection between various measures on the Urysohn space
  from one side and random matrices and theory random processes from other side looks very intriguing.
  Now we have very interesting volume on this topic, which is the result of the hard work of
  organizers of the conference on Urysohn space in Beer-Sheva at 2005 especially -  V.Pestov and A.Leiderman.

\section{Two theorems}
  In this paper we consider new kind of problems connected with general theory of metric spaces.
  Let us formulate two notions:{\it  approximation and globalization}.
  Let $\Bbb U$ be the Urysohn universal metric space, (see f.e.\cite{V})
  and let $ISO(\Bbb U)$ be the group of all isometries of $\Bbb U$:

\smallskip\noindent\textbf{Theorem A (Finite approximation of the group of isometries).}
   {\it The group $ISO(\Bbb U)$ contains everywhere dense locally finite subgroup.
   The same conclusion is true for the group of isometries of the rational universal space
   ${\Bbb Q\Bbb U}$.}

\smallskip

   By a partial isometry of any metric space $F$ we mean an
   isometry between two subsets of the space $F$. It is convenient
   to consider also "empty isometry" (= isometry with empty domain).
   Suppose a space $F$ is finite; each partial isometry is one-to-one map
   of a subset of $F$ to other subset of $F$ which preserve the metric
   It is easy to see that the set of all partial isometries
   generates with respect to superposition as a product
   {\it the inverse semigroup} which we call semigroup of partial isometries
    - with zero - empty isometry, and unity - identity map and denote it as
   $PISO(F)$. It is clear that this is a subsemigroup of symmetric inverse
   semigroup $SS_F$ of all one-to-one maps of the
   subsets of $F$ (see \cite{CP}). The group of all isometries $ISO(F)$ is a maximal subgroup
   of the semigroup $PISO(F)$ and is a subgroup of symmetric group $S_F$.
   For each subset $K\subset F$ we can consider the group of
   isometries $ISO(K)$ of subspace $K$ which is a partial subgroup
   of $PISO(F)$ in the natural sense.
   Now we can formulate the second theorem concerning
   to the important property of the finite metric spaces.

\smallskip\noindent
   \textbf{Theorem G (The group extension of partial isometries).}
   {\it  For every finite metric space $(F, \rho)$ there are exist
   a finite metric space $(F', \rho')$, and an isometric embedding $I:(F,\rho) \to (F',\rho')$:
    $\rho'(Ix,Iy)=\rho(x,y)$ such that

   {\rm 1.} (Extensions of partial isometries) each partial (in particular -global)
    isometry $h\in PISO(F)$ has an extension $T(h)$ which is the isometry of $F'$,
    e.g. $T(h)\in ISO(F')$ and
                  $$I(hx)=[T(h)](x);$$

   {\rm 2.}(Equiinvariance on the subgroups) those extensions have properties:
   for each $K\subset F$ (in particular for $K=F$) the restriction
   of $T$ on the group $I^*_K$ is a homomorphism of the group $ISO(K) \to ISO(F')$:
   $$T(g_1 g_2))=T(g_1)T(g_2), \forall g_1,g_2\in  ISO(K)$$}
   \smallskip

   In short, Theorem G asserts that we can
   simultaneously extend all partial isometries of a metric space
   $F$ and the group of its global isometries,  preserving the
   group structure of the partial subgroups of the isometries
   and in particulary of the group of all isometries of metric space $F$.
   Remark, that in general there are no any algebraic relations between
   extensions of isometries of the different subsets $K$, for
   example, $T$ is not homomorphism of the semigroup $PISO$ of all partial isometries
   to the group $ISO$. Our extensions are not unique and it is doubtful
   that the canonical  extension does exist.
   The theorem $G$ is true for the rational metric spaces (with rational distances)
   and of course it is enough to prove it for that case. Theorem A
   for the group of automorphisms of the universal graph was proved
   in \cite{MB,H,HL} (after \cite{Hr92}). The author in 2004 put the
   same - question about the existence of dense locally finite subgroup of
   the group $ISO(\Bbb U)$; and Professor P.Cameron attract my attention to those papers.

   Theorem B can be considered as a direct generalization of Hrushevski's \cite{Hr92}
   theorem about extensions of partial isomorphisms of the graphs
   with additional claim about equivariance of imbedding.
   This paper was written on 2005 and after that time several important
   publications were devoted to the same question - the paper \cite{Sol}
   in which the variant of the theorem $B$ was proved with completely
   different algebraic method and the paper \cite{P,P1} in which very useful
   survey and the simplest proofs in the same spirit as in \cite{Sol}
   was done. The elementary proof of the theorem in the spirit of
   the Hrushevski's  proof of globalization theorem for graphs will
   be done elsewhere.

   Now we start with some detailed elaboration of the main property
   of Urysohn space.

   \section{The strengthening of homogeneity in Urysohn space}

   The next theorem gives a refinement of the property of (ultra)homogeneity
   of the group $ISO(\Bbb U)$ of isometries of the Urysohn space $\Bbb U$.
   The condition of homogeneity asserts in particular that
   for any isometry $T$ of the finite subset $F \subset \Bbb Q\Bbb U$
   there exist an extension of $T$ to global isometry of $\Bbb Q\Bbb U$.
   We refine this assertion by an additional claim:

   \begin{theorem}
   {For any finite subset $F \subset \Bbb QU$
   there exists an isomorphic embedding $j:G \to ISO(\Bbb Q\Bbb U)$
   of the group $G = ISO(F)$
   of isometries of $F$  such that $(jg)(f)=g(f)$
   for all $g\in G$, $f\in F$. Consequently, the space $\Bbb Q\Bbb U$
   is a countable union of finite orbits of the group $jG$.}
  \end{theorem}

  \begin{proof}
   Denote by $r(\cdot,\cdot)$ the metric on the rational Urysohn
   space $\Bbb Q\Bbb U$. Our goal is to define a group extension of
   isometry group $G=ISO(F)$ onto whole space $\Bbb Q \Bbb U$.
   For this we will define $G$-orbit for each point
   $x \in \Bbb Q \Bbb U \setminus F $ and in this way to extend
   the action of the group $G$ to subgroup of isometries
   of $\Bbb Q\Bbb U$.

   Choose an arbitrary point $x \in \Bbb Q\Bbb U\setminus F$ and denote the
   space  $F_1=F\bigcup \{x\}$ as a separate metric space.
   So we have the following claim which is actually the simplest and the
   main case of the proof of the theorem G:
   Consider the finite metric space $F_1=F\bigcup \{x\}$, to find isometric extension
   $F'$ of the space $F_1$ and extend onto $F'$ the group of isometry $ISO(F)$ of $F$.
   If we can define such space $F'$ then, using the universality of $\Bbb Q\Bbb U$
   find the embedding $j$ of the space $F'$ to $\Bbb Q\Bbb U$ with
   condition: $j|_F=id$. By induction we can join consequently the points of
    $x \in \Bbb Q \Bbb U \setminus F $ and extent the group of isometry
   of the set $F$ onto whole space  $\Bbb Q\Bbb U$. Thus we reduce the problem
   to the question about finite metric spaces.

   Denote the distance vector of the point $x$ by  $a(x)\equiv a =\{a_f,\, f \in F\}$
   where $a_f=r(x,f),\, f \in F \}$ and consider its $G$-orbit in the space of
   rational vectors: $ga=\{a_{gf}=r(x,gf),  f \in F\}; g\in G$. Suppose $G_0$ is a
   stable subgroup of group $G$ of the vector $a$, so
   $G_0=\{h\in G: \forall f \in F, r(x,hf)=r(x,f)$, (in another words $G_0$ is
   intersection $ISO(F)\bigcap ISO(F\cup \{x\}$). Then we can be identified
   the orbit of $a$ with the right cosets  $G/G_0$. We will denote
   the classes (or the orbit of the vector $a$) as $x\equiv x_0\thicksim a$
   and $x_1, \dots x_k, k=|G/G_0|$ with the points of the
   The needed space $F'$ will be the union $F \bigcup \{x_0=x,x_1 \dots x_k\}=
   F_1 \bigcup \{x_1,\dots x_k$. Define the action of the group $G$ on
   $F'=F \setminus \{x_0,x_1, \dots x_k\}$ as natural action on the right classes
   $G/G_0$. This is the extension of the action of the group $G$ onto space $F'$
   Our task is to extend the metric and to prove that the extension action is isometric.

   Remark, that by definition, $r(xg,f)=r(x,g^{-1}f)$, and
   $r(xg,gf))=r(x, gg^{-1}f)=r(x,f)$, for all $f \in F$. So it is
   enough to define the distances between points $x_s$ - $r(x_i,x_j)$, or in order to
   save isometric action we must have $r(xg,xh)=r(x,hg^{-1}x$,
   so we have to define only distances between $x_0 \equiv x$ and points
   $x_s, s=1 \dots k$.
   The distances $r(x, x_s), s=1\dots $ must satisfy to triangle inequality for
   any triple of points of type ($x, x_s, f)$, thus we have the following
   condition:
   $\forall f\in F \quad |r(x,f)-r(xg,f)|\leq r(x,xg)\leq r(x,f)+r(xg,f)$,
   or because $r(xg,f)=r(x,g^{-1}f)$ for some $g \in G$, we have

   $$|r(x,f)-r(x,g^{-1}f)|\leq r(x,xg) \leq r(x,f)+r(x,g^{-1}f)$$
   for all $f \in F$.
   In order to have such possibility we must prove the following inequality:
   $$max_{f \in F}|r(x,f)-r(x,gf)|\leq min_{f \in F} r(x,f)+r(x,gf).$$

   But this is true: choose any $f_1,f_2$, then

         $$|r(x,f_1)-r(x,gf_1)|\leq r(x,f_2)+r(x,gf_2),$$
   indeed
           $r(x,f_1)\leq r(x,f_2)+r(f_1,f_2)\leq
   r(x,f_2)+r(gf_1,gf_2)\leq r(x,f_2)+r(x,gf_1)+r(x,gf_2)$.

   Consider the maximum of the left side over $f_1$ and minimum of the right side
   over $f_2$ of the previous inequality, then we obtain needed inequality.
   So, the distance can be defined correctly. For example, we can define
   the distances $r(x, x_s), s=1\dots $, as the left side of the inequalities:
   $$r(x,xg^{-1})\equiv max_{f\in F} |r(x,f)-r(x,gf)|.$$
   It is evident that by this the definition the distance is the uniform
   norm in $C(F)={\Bbb R}^{|F|}$ of the differences:
   $r(x,xg^{-1})=||r(x,\cdot)-r(x,g\cdot)||$. From this we can conclude immediately
   that triangle inequality is true for all triples of type
   $(x,gx,hx)$ because it follows from triangle inequality for the norm,
   and from the invariance of the norm with respect to action of group $G$:

   $$r(x,xg^{-1})+r(x,xh^{-1})=||r(x,\cdot)-r(x,g\cdot)||+||r(x,\cdot)-r(x,h\cdot)||
    \geq ||r(xg,\cdot)-r(xh,\cdot)||=$$
   $$ =||r(x,\cdot)-r(xhg^{-1},\cdot)||=r(x,xgh^{-1}).$$
   So we have checked triangle inequality for all types of triples,
   and the invariance of the extended metric $r$ under the action of the
   group G.
\end{proof}

   The way to define the distance which we use here is canonical but
   not unique and even not convenient for the general case.

   The next corollary immediately follows from the theorem and
   universality of $\Bbb U$.

\begin{corollary}
  {For each finite metric space $F$ there exist an equivariant
  isometric imbedding $i$ of $F$ to the space $\Bbb U$, which means
  that there exists imbedding $j$ of the group $ISO(F)$ to the group
  $ISO (\Bbb U)$ such that $(i(gf)=j(g)if)$ for each $f \in F$ and $g
  \in ISO(F)$.}
\end{corollary}

\smallskip\noindent
  {\bf  Remark.} The above theorem and the corollary allow to give
   new proof of the theorem by V.~V.~Uspensky \cite{U}: each Polish
   space $F$ can be embedded into the universal space $\Bbb U$
   in an equivariant way: $I:(F, \rho')\to (\Bbb U, r)$, $I^*:ISO (F) \to ISO(\Bbb U)$,
   where $I$ is an isometry,  $I^*$ is a monomorphism of groups, and
    $$I(gf)=I^*(g)I(f), \quad f \in F,\; g \in ISO(F).$$

  Indeed, the proof of the theorem 1 include the claim of Uspensky'
  theorem for finite metric spaces; in order to extend it to the
  arbitrary metric spaces it is enough to mention that
  our construction is continuous in the natural sense: two
  closed isometries of the space $F$ imbedded to closed isometries
  of $\Bbb U$.
\section{Proof of equivalence theorems A and G}

    First we deduce Theorem G using the Approximation Theorem A.

\begin{theorem} \textbf{$A\implies G$}.
\end{theorem}

   \begin{proof}
   {It is enough to prove the theorem G for rational metric spaces:
   suppose the Theorem A is proved for the universal space $\Bbb Q\Bbb U$
   over field $\Bbb Q$ of rational numbers.

   The assertion of Theorem A means that there exist increasing sequences of finite
   subgroups $G_n \in ISO (\Bbb Q\Bbb U)$ with union which is dense in
   the group: $ISO(\Bbb Q\Bbb U)  = \textbf{Cl} \cup_{n=1}^{\infty}G_n$.
   For each $n$ choose a set $E_n$ of points $x \in (\Bbb Q\Bbb U)$ and consider the set $F_n$
   which is the $G_n$-orbits of $E_n$ such that the action of $G$ on $F_n$ is
   faithful (each nonidentical element of $G_n$ has nonidentity action).
   We can assume that the sequence of the sets $F_n$ increases.
   Because of density of $\cup_{n=1}^{\infty}G_n$ in the group $ISO(\Bbb Q\Bbb U)$
   it is clear that $\cup F_n=\Bbb Q\Bbb U$ and therefore for any $g \in ISO(\Bbb Q\Bbb U)$
   and natural $N$ there exists $n(N)$ and $g' \in G_N$ such that $g(f)=g'(f)$ for all
   $f \in F_n$. In particular we can assume that the group $ISO(F_n)$ contains in $G_n$.

      Let us start with an arbitrary finite metric space $F$; since of universality
   we can embed the space $F$ isometrically into the universal space
   $\Bbb Q \Bbb U$, and regard $F$ as a subset of $\Bbb Q\Bbb U$. So $F \subset F_n$
   for some $n$. By theorem 1 we can embed $F$ in such way that $ISO(F)$ is a subgroup
   of $ISO(\Bbb Q\Bbb U)$.
   Each partial isometry $T$ of $F$ are partial isometries of $F_n$ and by the main
   property of universal space each partial isometry can be extended to the global
   isometry $\bar T$ of $\Bbb Q\Bbb U$. Because of finiteness of $F$ we can choose
   $m$ such that extension $\bar T$ for all partial isometries $T$ acts on $F_m$
   as some elements of the group $G_m$. So on $F_m$ we have isometrical extensions of all
   partial isometries of $F$ and in the same time the group of isometries
   of $F$ is subgroup of $G_m$. So we find a finite metric space $F_m$ which
   satisfies to all conditions of the theorem G.}
   \end{proof}

   Now we deduce Theorem A (Approximation Theorem) using Theorem G.

\begin{theorem} \textbf{$G \implies A$.}
\end{theorem}

   \begin{proof}
   {Again it is sufficient to prove Theorem A for the  rational universal space $\Bbb
   Q\Bbb U$, because the rational universal Urysohn space $\Bbb Q\Bbb U$ is
   a countable dense subset of $U$ and the group of isometries $ISO(\Bbb QU)$ is an
   uncountable and dense subgroup in $ISO(\Bbb U)$ (see \cite{CV}).
   Thus we must prove the existence of a locally finite subgroup in $ISO(\Bbb Q\Bbb U)$.

   Enumerate all points of $\Bbb Q\Bbb U$ as $x_1,x_2, \dots $. Choose
   $F_1=\{x_1\}$ and consider the identity isometry as the group $G_1 =\{{\rm Id}\}$.
   Suppose we already have an increasing sequence of the metric spaces
   $F_1 \subset F_2\subset \dots \subset F_n \subset \Bbb Q\Bbb U$  such that for each $k<n+1$,
   $x_i \in F_k; i=1, \dots k$ and each partial isometry
   of $F_k$, can be extended in $F_{k+1}$ to an isometry of $F_{k+1},  k=1, \dots, n$
   and the group of isometries of $F_k$ embed isomorphically to $ISO(F_{k+1})$
   accordingly to the theorem $B$.

   Consider now a metric space $F'=F_n \cup \{x_{n+1}\}$
   and apply to it Theorem $G$, which gives a finite metric space $F_{n+1}$
   in which all partial isometries of $F'$ can be extended to an isometry
   of $F_{n+1}$. We can embed $F_{n+1}$ as a subset of
   $\Bbb Q\Bbb U$ and extend an isometry of $F_{n+1}$ to the whole space.
   It is clear that the union $\cup_n F_n$ is the whole space $\Bbb Q\Bbb U$ and
   $\cup G_n=G_{\infty}$ is a locally finite subgroup of $ISO(\Bbb
   Q\Bbb U)$. The density of $G_{\infty}$ in  $ISO(\Bbb Q\Bbb U)$ (in the weak
   topology) is an easy consequence of the construction.}
\end{proof}

\smallskip
\noindent\textbf{Remark.}
   1. In the proof of Theorem A we used the existence of the universal space,
   only for enumeration of all finite metric spaces, it is obviously
   possible to change the proof so that to obtain a new proof of
   the existence of the universal space  simultaneously with Theorem $A$.

\smallskip
   2. The equivalence between the theorems $A$ and $G$ is useful itself because
   of interest of the question for what classes of universal objects
   both assertions are true or not. Sometime it could be easy to establish or
   to disprove one of them and consequently to establish or to disprove
   another. As in the case of category of graphs we will prove theorem $G$ for
   metric spaces and consequently establish theorem $A$ for the space $\Bbb U$.

We conclude this section with the following

\textbf{Conjecture.}{\it The group $ISO(\Bbb U)$ (as well as the group of automorphisms of the
universal graph $\Gamma$) contains well-known Hall's group as a dense subgroup.} Recall that Hall
group is the universal group in the class of all locally finite groups. It can be described as
inductive limit of the symmetric groups $S_n, n \geq 3$ with respect to embedding:
$$S_n \rightarrow S_{n!}, \quad n=3, \dots.$$ (the left shift in $S_n$ on the element $g \in S_n$ is considered as
permutation of $(1,\dots ,n)$.

\section{Distance matrices, criteria of Urysohnness and various constructions of Urysohn space}

We want to give some comments about distance matrices and construction of Urysohn space.

  1.\textbf{Admissible Lipschitz functions and admissible vectors}.
  It makes sense to systematize here the main characteristic properties
  of Urysohn space in terms of approximation of Lipschitz functions
  and to give various constructions of it.
  Let us call the Lipschitz function $u$ {\it admissible} if it satisfies to the inequalities:
  $|u(x)-u(y)|\leq \rho(x,y)\leq u(x)+u(y)$. The function $f$ on the metric space $(X,\rho)$
  is called {\it a  distance function} if it has a form of type $f(.)=\rho(x,.)\equiv f_x(.)$.
  It is clear that such a functions has Lipschitz norm 1.

 \begin{theorem} {The Polish space $(X,\rho)$ is isometric to Urysohn space
  $\Bbb U$ iff each admissible Lipschitz function $u(.)$ is a pointwise limit
  of the sequence of distance functions $f_{x_n}(.)$ for some sequence $\{x_n\}$ e.g.
  $u(z)=\lim_n f_{x_n}(z)$ for all $z \in X$.}
\end{theorem}
 In another words it means that the set of distance functions are weakly dense in the unit
 sphere of the of Lipschitz space with respect to duality with Kantorovich-Rubinstein space
 ( or equivalent duality) -see \cite{MPV}.
  The theorem immediately follows from the {\it criteria of universality}
 below which was done in \cite{V} and previous papers of the author which were mentioned in
  \cite{V}). The criteria formulated in the terms of notion of distance matrices:
  \begin{theorem}
  Let $(X,\rho)$ is a Polish space, and $\{x_i\}_{i=1}^{\infty}$ is an arbitrary
  everywhere dense sequence in $X$. Then $(X,\rho)$ is isometric to Urysohn space
  $\Bbb U$ iff for each $n\in \Bbb N$, $\epsilon >0$ and real positive vector
  $\{a_i\}_{i=1}^n$, such that
  $$|a_i-a_j|\leq \rho(x_i,x_j)\leq a_i+a_j \footnote{such vectors were called {\it admissible
  vector} for distance matrix $\{\rho(x_i,x_j)\}_{i,j=1}^n$.} $$
  there exist $k>n$ for which
  $$\max_{i=1,\dots n} |a_i-\rho(x_i,x_k)|<\epsilon.$$
  If the condition is valid for some everywhere dense sequence $\{x_n\}$ then it is true
  for all such sequences.
  \end{theorem}
  The distance matrix $\{r_{i,j}\}(=\{\rho(x_i,x_j)\})$ which satisfy to the
  condition of the theorem called {\it universal distance matrix}.
  So, universality of the matrix $\{\rho(x_i,x_j)\}$ does not depend
  on the choice of everywhere dense sequences $\{x_i\}$.

  In its turn the criteria is adaptation of the original ideas of P.Urysohn \cite{U}:
  {\it the Urysohn construction gives the universal rational distance matrix} -
  in this case we did not need in the previous formulation in $\epsilon>0$.
  For the rational case the proof of the criteria is a standard for model theory:
  we use amalgamation property for finite metric space, back and forth method and so on.
  This gives the {\it inductive} construction of the universal distance matrix and
  correspondingly universal Urysohn space (cf. \cite{V}). Remark that such construction is nothing
  more than reformulation of the construction of Fraisse limit which appeared many years later and
  which is simply inductive limit of the finite metric spaces in the traditional terminology (see \cite{Ku}).
  But the difference between classical model theory and our case
  consist in the following: after construction of the inductive limit we must make the completion
  of the space, I do not see how to axiomatize this construction for uncountable case.

   2.\textbf{Randomness and universality}. From the description of the universal distance matrices we can conclude that
   this matrices generate the everywhere dense $G_{\delta}$-set (generic) in the space of all distance matrices.
   It means that Urysohn space is generic in the appropriate sense (see\cite{V}) in the category of all
   Polish spaces. It leads to the remarkable fact that random (in the very wide sense) choice of limiting finite spaces
  also gives us the Urysohn space {\it after completion!} - it is not true without completion
  because even the set of values of the distances could be different and the limits itself are not
  isomorphic but completion erase this difference. This allows us to say that in a sense the randomness implies
  the universality. For countable case this effect is well-known (see f.e.\cite{PR}),
  and even uses in terminology -"the universal graph" people called incorrectly "random graph" which is not graph
  but the measure on the set of all graphs; nevertheless there is no ambiguity because with
  probability one (with respect to the distinguish measures!) the random graph is universal.
  But in the more complicate cases and in particular for continuous spaces this effect is more
  subtle: the role of the choice of measure or in another words, the sense of "randomness", - is very essential,
  it is not clear whether there is the distinguish measure (like uniform measure) \footnote
  {f.e., the probability measures on the set of countable universal triangle free graphs
  which are invariant with respect to the group of infinite permutations of the vertices,
  had been constructed recently by F.Petrov and author using ideas from the paper
  \cite{V1}}.

  3.\textbf{General Constructions}. Returning to the construction of Urysohn space recall that M.Katetov's \cite{K}
  (and later M.Gromov's \cite{G}) gave another construction, which is in fact also based on theorem
  above in continuous version. Namely, Katetov pointed out on the concrete and useful realization of the
  one-point extension of the given metric space $(X,\rho)$ -  $(X \cup {y} \equiv X', \rho')$
  such that the given admissible Lipschitz function $u$ on $X$ (with Lipschitz norm 1)
  becomes the distance function (in the sense of the definition above) on the space $(X'\rho')$:
  $u(x)=\rho'(y,x)$. This realization is the following: consider isometric embedding of the
  metric space $(X,\rho)$ to the space of bounded continuous functions ${\bar C}(X)$:
  (in terminology of \cite{MPV} this is the Hausdorf-Kuratowski embedding)
  $$HK:x\mapsto \rho(x,.)-\rho(x_0,.) \in {\bar C}(X).$$ Here $x_0$ is an arbitrary point of $x$
  \footnote{it is better to embed the space of formal differences $\delta_x-\delta_y$ to the
  ${\bar C}(X)$}: $HK(\delta_x-\delta_y)=\rho(x,.)=\rho(y,.)$ instead of choice an arbitrary point $x_0$
  Then as  it easy to check, needed metric space
  $(X'\rho')$ is the union of the image of the space $X$ and function $u_1(.)=u(.)-\rho(x_0,.)$ -  $HK(X)\cup u_1$ with
  sup-norm:$\|u_1(.)-HK(x)(.)\|=\|u(.)-\rho(x,.)\|=u(x)$.
  The inductive procedure which used this operation allows to construct after completion the
  universal Urysohn space. We use above some kind of the same trick in our proof of existence
  of the extension of the space in which the special partial isometry can be extended up to global
  isometry.

  4.\textbf{Shift-invariant constructions}. In the paper \cite{CV} the construction of Urysohn space
  was done with additional assumptions: which shortly can be describe as follow:

Consider an invariant metric $r$ on the group $G$ which is either $\Bbb Z$ or $\Bbb R$, and denote
$\phi(.)$ where $r(x,y)=\phi(x-y)$ the positive continuous function on the group $G$ with property
which is equivalent to be a metric for $r$:

$$\phi(0)=0,\quad \phi(g)>0\quad (g \ne 0), \quad \phi(g)=\phi(-g),\quad \phi(g_1+g_2)\leq
\phi(g_1)+\phi(g_2).$$ The completion of the group $G$ under an
invariant metric called {\it monothetic metric group}. This is of
course the group; we denote it $G^r$; $G^r$ is the abelian group
with the invariant metric. Subgroup $G \in G^r$ is by definition
everywhere dense in $G^r$ and the shift is isometry with dense orbit
($G$). Thus the group $G$ as the group of shifts is the subgroup of
the group $ISO(G^r)$ of all isometries of $G^r$ (see \cite{CV}) Many
monothetic groups are known in ergodic theory, harmonic analysis
(f.e. so called Kronecker group etc.) Recall that metric on a space
$X$ called universal if the space $X$ or its completion with respect
to that metric is isometric to the Urysohn space. Remark that the
invariant metrics on the group $G$ is very meagre in the set of all
metric on $G$ nevertheless the following fact is true:
 \begin{theorem} (\cite{CV})
There is an invariant universal metric on the group $G$, moreover, the set of universal invariant
metrics is generic (everywhere dense $G_{\delta}$-set) in the set of all invariant metrics on $G$
equipped with a natural topology.In another words; there exists monothetic metric group which are
isometric to Urysohn space as metric space, more directly? the Urysohn space can have the structure
of monothetic (abelian) metric group for $\Bbb Z$ as well as for $\Bbb R$. There are uncountably
many of such structures on the Urysohn space which are mutually not isomorphic as a metric groups.
 \end{theorem}

 The problem of the isomorphism of that group as topological group is open.

Remember that locally compact monothetic groups are compact (f.r. torus or group of integer p-adic
numbers etc.). The proof of the existence of such metric is far from evidence and was proved in
\cite{CV} with using of the fact which was proved in the paper \cite{C} about integer metric
spaces. Here we formulate the interesting property of the set of finite dimensional distance
matrices which gives more transparent proof of the existence of universal invariant metric in terms
of admissible sets.

Our goal is to construct universal Toeplitz distance matrix $$\{r_{i-j}\}_{i,j} \quad i,j \in \Bbb
Z.$$ The problem is:  we can not construct the distance matrix as in general case, adding a new
points into the space (or column of the distance matrix) in arbitrary way because we must satisfy
to the Toeplitz condition. So we can add a new column only after several intermediate steps. For
this we can use the following finite-dimensional lemma.

Lemma ("billiard lemma"). For each $n \in \Bbb N$ consider a distance matrix of order $n$ -
$M=\{r_{i-j}\}_{i,j} \quad i.j =1 \dots n$, define the set $Adm M$ of admissible vectors for $M$:
$$Adm(M)=\{a \in {\Bbb R}^n: a_i-a_j\leq r_{i-j}\leq a_i+ a_j, i,j=1,\dots n \}\footnote{see footnote 1}$$
and let $x, y \in Adm(M)$ - two vectors from $Adm(M)$.
 Then there exist a sequences of the vectors $$a^n=x,a^{n+1},a^{n+}2 \dots a^k,\dots a^N=y$$ which are belong
 to $Adm(M)$ and such that for each $$k: a^{k+1}_i=a^k_{i-1}, i=2,\dots n, k=n, \dots N-1.$$

 \begin{proof}
We will start with the case $n=2$. In this case the set $Adm$ depends on one positive number -
$r_{1,2}=r$, and $$Adm_r=\{(a_1,a_2): |a_1-a_2|\leq r \leq a_1+a_2\}$$. Suppose we have two vectors
$x,y\in Adm $, define the sequence of vectors as a vertices of the "billiard" piece-wise linear
trajectory which belongs to $Adm$, and whose edges are alternatingly parallel to the first or
second axis; the vertices of the path belongs either to the diagonal $X=Y$, or the boundary of
$Adm(M)$ - $|X-Y|=r$. In order to describe the vectors we have to compare  $x_1$ and $y_2$, without
diminishing of the generality we can assume that $x_1 geq y_2$, and suppose that
$y_2+kr<x_1<y_2+(k+1)r, k \in N$.  Define the chain:
$$a_3=(x_1,x_2)\mapsto a_4=(x_1,x_1)\mapsto a_5=(y_2+k r,x_1)\mapsto a_6=(y_2+k r,y_2+k r)\mapsto$$
$$\mapsto a_7=(y_2+(k-1)r,y_2+k r)\dots (y_2+r,y_2+r)\mapsto(y_2,y_2+r)\mapsto(y_2,y_2)\mapsto
a_N=(y_1,y_2).$$ This chain defines piece-wise linear path with segment between two adjacent
vectors so we obtain the trajectory like billiard orbit. However the even vectors $a_{2k}$ plays
only geometrical role (billiard trajectory), so the sequences of columns of the matrix are the
vectors
$$(x_1,x_2),(y_2+k r,x_1),(y_2+(k-1)r,y_2+k r),...(y_2,y_2+r),(y_1,y_2),$$ this links the vector
$x=(x_1,x_2)$ and vector $y=(y_1,y_2)$ in appropriate (Toeplitz) way: each the second coordinate of
the vector is equal to the first coordinate of the previous.

 If $x_1 \leq y_2$, say, $y_2-(k+1)r \leq x_1\leq y_2-kr$, then we can use the same procedure
 with intermediate values of coordinates: $y_2-kr,y-(k-1)r etc.$.

 Let us consider the general $n$; we have two vectors $(x_1, \dots x_n),(y_1, \dots y_n) \in Adm(M)$,
 where $M$ is the distance
 matrix of order $n$. The sequence of the vectors $a^k$ defines in the same way: compare coordinates
 $x_1$ and $y_n$, if the vector $(y_n,x_1,\dots x_{n-1}) \in Adm(M)$ - is admissible: then consequently
 we have chain $a_n=(x_1, \dots x_n)\mapsto a_{n+1}=(x_1,x_1,x_2 \dots x_{n-1})\mapsto a_{n+2}=
 (y_n,x_1,\dots x_{n-1})\mapsto \dots\mapsto ,(y_1,y_2,\dots y_n)$. In opposite case (if
 $(y_n,x_1,\dots x_{n-1}) \in Adm(M)$ is not admissible) the crucial role plays again the inequality between
 $x_1$ and $y_n$ and again intermediate values of $y_n -...$ give the needed sequences of the vectors $a^k$.
 Let us call the sequence of vectors $a^k$ which combine the vectors $x$ and $y$ - "billiard
 fragments". Now we can construct the universal Toeplitz distance matrix in the same way as in the Urysohn
 construction: choose the sequence of all admissible vectors and include in between of two neighbors
 if necessary the billiard fragments. The criteria of universality is valid for this sequence and Toeplitz
 property is true by definition.
 \end{proof}


\begin{thebibliography}{11}
\bibitem{U}P.~S.~Urysohn. Sur un espace metrique universel.
Bull.Sci.Math. 1927, v 180, 1-38.

\bibitem{BV} V.~N.~Berestovsky and A.~Vershik). Manifolds with intrinsic metric, and nonholonomic spaces
In "Representation theory and dynamical systems". AMS. Adv. Sov. Math.  9, 253-267 (1992).

\bibitem{C}P.~Cameron, Homogeneous Caley objects, European J.Combin. 21,745-760 (2000).

\bibitem{CV}P.~Cameron and A.~Vershik,
Some isometry groups of Urysohn spaces. Ann. Pure and Appl. Logic
v.143 no. 1-3, 70-78 (2006).

\bibitem{V^0}A.~Vershik, The universal Uryson space, Gromov's metric triples, and random metrics
on the series of natural numbers. Uspekhi Mat. Nauk 53, No.5, 57-64 (1998). English translation:
Russian Math. Surveys 53, No.5, 921-928 (1998).

\bibitem{V}A.~Vershik. Random metric spaces and universality
Russ. Math. Surv. v.2(356), 2004, 65-104.

\bibitem{Hr92}E.~Hrushevski. Extended partial isomorphisms of
graphs. Combinatorica 12 (4)(1992) 411-416.

\bibitem{MB}D.~McPherson,R.~Bhtattacharia. Locally finite group
acting on the graph. Preprint. Univ. of Leeds, 2003.

\bibitem{U} V.~V~Uspensky. The Urysohn universal metric space is homeomorphic
to a Hilbert space. Topology Appl. 139 (2004), no. 1-3, 145--149.

\bibitem{P} V.~Pestov. The isometry group of the Urysohn space as a Le\'vy group.
Topology Appl. 154 (2007), no. 10, 2173--2184.

\bibitem{P1} V.~Pestov. A theorem of Hrushovski-Solecki-Vershik applied to uniform
and coarse embeddings of the Urysohn metric space, arXiv:math/0702207.

\bibitem{CP} A.~H.~Clifford and G.~B.~Preston. {\it The
Algebraic Theory of Semigroups}, Vol. I. Amer. Math. Soc.,
Providence, RI, 1961.

\bibitem{K} M.~Kat\v{e}tov, ``On universal metric spaces,''
in: {\it  General Topology and Its Relations to Modern Analysis and
Algebra}, VI (Prague, 1986), Helderman, Berlin, 1988, pp. 323--330.

\bibitem{G} M.~Gromov, {\it Metric Structures for Riemannian and
Non-Riemannian Spaces}, Birkh\"auser Boston, Boston, 1999.

\bibitem{HL} B.~Herweg, D.Lascar. Extending partial automotrphisms
and the profinite topology on free groups. TAMS, v. 352,N.5 (1999).
p.1985-2021.

\bibitem{H} B.~Herwig. Extending of partial isomorphisms on finite
structures. Combinatorica. 15(3)(1995), p.365-371,.

\bibitem{Sol} S.Solecki. Extending partial isometries. Israel J.
Math. 150 (2005), 315-332.

\bibitem{V1} A.~Vershik. Classification of Measurable Functions of Several
Variables and Invariantlly Distributed Random Matrices. Func.Anal
and Its Appl.36 (2)(2002 12-27.

\bibitem{Ho} R.~Holmes, The universal separable metric space of Urysohn
and isometric embeddings thereof in Banach spaces, {\it Fund. Math.}, {\bf140}, No.~3, 199--223
(1992).

\bibitem{MPV} J.Melleray, F.Petrov, A.Vershik.
 Linearly rigid metric spaces and the embedding problem. Submitted
 to Fundamenta Math.

\bibitem{Ku} W.Kubis. Fraisse limit - categorical approach. Preprint.
Beer-Sheva Univ.Israel.(2005)

\bibitem{PR} M.Pouzet,B.Roux. Ubiquity in category for metric spaces
and Transition systems. European J.Combin. 17(1996) no 2-3 291-307.

\end{thebibliography}
\end{document}